\documentclass[10pt]{amsart}

\usepackage{amsfonts} 
\usepackage{amssymb}
\usepackage{amsmath} 
\usepackage{amsthm} 
\usepackage{stmaryrd}
\usepackage{mathtools}
\usepackage{xcolor}
\usepackage{bm}
\usepackage[normalem]{ulem}
\usepackage{enumitem}

\usepackage{fancyhdr}
\newcommand{\cc}{\mathfrak{c}}
\newcommand{\C}{\mathbb{C}}

\newcommand{\F}{\mathbb{F}}

\newcommand{\isoto}{\xrightarrow{\sim}}

\newcommand{\m}{\mathfrak{m}}

\newcommand{\OK}{\mathcal{O}}
\newcommand{\p}{\mathfrak{p}}
\newcommand{\Pp}{\mathfrak{P}}
\newcommand{\proj}{\mathbb{P}}
\newcommand{\Q}{\mathbb{Q}}
\newcommand{\R}{\mathbb{R}}
\newcommand{\T}{\mathbb{T}}

\newcommand{\Z}{\mathbb{Z}}

\DeclareMathOperator{\Div}{Div}
\DeclareMathOperator{\End}{End}

\DeclareMathOperator{\Frob}{Frob}
\DeclareMathOperator{\Gal}{Gal}

\DeclareMathOperator{\GL}{GL}

\DeclareMathOperator{\Res}{Res}

\DeclareMathOperator{\SL}{SL}

\newcommand{\Cl}{\mathrm{Cl}}
\newcommand{\Eis}{\mathrm{Eis}}

\newcommand{\sm}[4]{\ensuremath{\big(\begin{smallmatrix}#1 & #2 \\ #3 & #4\end{smallmatrix}\big)}}

\newcommand{\ttmat}[4]{\left( \begin{array}{cc}
#1 & #2 \\
#3 & #4
\end{array}
\right)}

\usepackage{graphicx,stackengine,xcolor}
\newcommand\Circle[1]{%
  \def\useanchorwidth{T}%
  \def\stacktype{L}%
  \stackon[0pt]{#1}{\scalebox{2.8}[1.15]{{$\bigcirc$}}}%
}

\definecolor{OliveGreen}{rgb}{0.0, 0.6, 0.0}
\definecolor{DarkBlue}{rgb}{0.0, 0.0, 0.8}

\colorlet{prestoncolor}{red!100}
\colorlet{jackiecolor}{OliveGreen!100}

\newcommand\pout{\bgroup\markoverwith{\textcolor{prestoncolor}{\rule[0.5ex]{2pt}{0.4pt}}}\ULon}
\newcommand\jout{\bgroup\markoverwith{\textcolor{jackiecolor}{\rule[0.5ex]{2pt}{0.4pt}}}\ULon}

\usepackage{hyperref}
\hypersetup{
    colorlinks=true,
    linkcolor=blue,
    urlcolor=blue, 
    citecolor= blue}
\usepackage[noabbrev, capitalize, nameinlink]{cleveref}

\newtheorem{theoremA}{Theorem}

\newtheorem{theorem}{Theorem}[section]
\newtheorem{corollary}[theorem]{Corollary}
\newtheorem{proposition}[theorem]{Proposition}
\newtheorem{lemma}[theorem]{Lemma}

\theoremstyle{definition}

\newtheorem{examplex}[theorem]{Example}
\newenvironment{example}
  {\pushQED{\qed}\examplex}
  {\popQED\endexamplex}

\theoremstyle{remark}
\newtheorem{remarkx}[theorem]{Remark}
\newenvironment{remark}
  {\pushQED{\qed}\remarkx}
  {\popQED\endremarkx}

\newtheoremstyle{dotless}{}{}{}{}{\itshape}{}{ }{}
\theoremstyle{dotless}

\begin{document}

\title{A modular construction of unramified $p$-extensions of $\Q(N^{1/p})$}

\author{Jaclyn Lang}
\address{Temple University}
\email{jaclyn.lang@temple.edu}

\author{Preston Wake}
\address{Michigan State University}
\email{wakepres@msu.edu}

\subjclass[2010]{11F33, 11F80, 11R29, 11R37}
\keywords{Eisenstein ideal, class group, Galois representation}

\begin{abstract}
We show that for primes $N, p \geq 5$ with $N \equiv -1 \bmod p$, the class number of $\Q(N^{1/p})$ is divisible by $p$.  Our methods are via congruences between Eisenstein series and cusp forms.  In particular, we show that when $N \equiv -1 \bmod p$, there is always a cusp form of weight $2$ and level $\Gamma_0(N^2)$ whose $\ell$-th Fourier coefficient is congruent to $\ell + 1$ modulo a prime above $p$, for all primes $\ell$.  We use the Galois representation of such a cusp form to explicitly construct an unramified degree $p$ extension of $\Q(N^{1/p})$.
\end{abstract}

\maketitle

\section{Introduction}
Throughout this paper, $N$ and $p$ denote prime numbers such that $p \geq 5$.  

\subsection{Main results} 
We give a proof of the following theorem via congruences between Eisenstein series and cuspidal modular forms.
\begin{theoremA}\label{main thm}
If $N \equiv -1 \bmod p$ then $p$ divides the class number of $\Q(N^{1/p})$.
\end{theoremA}
This theorem was conjectured by Kobayashi \cite[Conjecture 1]{Kobayashi16}, who proved it in the case $p=5$.  
A proof of \cref{main thm} using Galois cohomology was sketched by Calegari on his blog \cite{Frank}.  Calegari asks whether there is a direct proof of \cref{main thm} and whether there is ``an easy way to construct the relevant unramified extension of degree $p$".  The purpose of this paper is to do exactly that. 

We give an explicit construction of the corresponding unramified extension of degree $p$ of $\Q(N^{1/p})$ using the Galois representation of a modular form.  
Explicitly, we prove the following.

\begin{theoremA}\label{main thm 2}
Assume that $N \equiv -1 \bmod p$. 
\begin{enumerate}[label=(\alph*)]
\item\label{congruence} There is an newform $f$ of weight $2$ and level $\Gamma_0(N^2)$ and a prime ideal $\p$ over $p$ in the ring of integers $\mathcal{O}_f$ of the Hecke field of $f$ such that for all primes $\ell$,
\begin{equation}
\label{eq:eis congruence}
a_\ell(f) \equiv 1+\ell \bmod \p.
\end{equation}
\item\label{character} Moreover, if $s$ is the largest integer such that $a_\ell(f) \equiv 1 + \ell \bmod \p^s$ for all primes $\ell \neq N$ and $t_f: \Gal(\overline{\Q}/\Q) \to \mathcal{O}_{f, \p}$ denotes the trace of the Galois representation of $f$, then 
\[
t_f|G_{\Q(N^{1/p})}  \equiv \chi \epsilon + \chi^{-1} \bmod{\p^{s+1}},
\]
where $\epsilon$ is the $p$-adic cyclotomic character and $\chi : \Gal(\overline{\Q}/\Q(N^{1/p})) \to (\mathcal{O}_{f, \p} / \p^{s+1})^\times$ is a non-trivial everywhere unramified character with $\chi \equiv 1 \bmod{\p^s}$. 
\end{enumerate}
\end{theoremA}

The values $1 + \ell$ on the right hand side of \eqref{eq:eis congruence} {are} the Hecke eigenvalue{s} of an Eisenstein series, so we say that a form $f$ satisfying \eqref{eq:eis congruence} for all primes $\ell \neq N$ is \emph{congruent to an Eisenstein series}, and say \eqref{eq:eis congruence} is an \emph{Eisenstein congruence}, modulo $p$.

\begin{remark}

For level $\Gamma_0(M)$ with $M$ prime, Mazur \cite{Mazur77} gave a necessary and sufficient condition for an Eisenstein congruence to exist. For squarefree $M$,  partial necessary and sufficient conditions have been found by Ribet \cite{Ribet10} (see also \cite{Yoo17}) using geometry of modular Jacobians, and by the second author and Wang-Erickson \cite{WWE21} using Galois deformation theory.  For some nonsquarefree $M$, sufficient conditions have been proven by Martin \cite{Martin17} using Jaquet--Langlands theory and necessary conditions by Yoo \cite{Yoo19} using geometry of modular Jacobians. However, both those works do not consider the case where $M$ is a square of a prime. The case where $M$ is a square of a prime has been considered by Gross--Lubin \cite{GL86} and Calegari \cite{Calegari06}, but then only when $p \mid M$.
\end{remark}

Note that \cref{main thm 2} implies \cref{main thm} by class field theory,  because $\chi$ cuts out a degree-$p$ unramified extension of $\Q(N^{1/p})$.  \cref{main thm 2} may be thought of as an explicit version of \cref{main thm} because the Fourier coefficients of the newform $f$ can often be efficiently computed. Extra information about the class field of $\Q(N^{1/p})$ can be read off from this data, as we illustrate in \cref{example 5 19}.

\subsection{The $N \equiv 1 \bmod p$ case}\label{comparison}
To understand the context for Theorems \ref{main thm} and \ref{main thm 2}, 
we recall what is known in the case when $N \equiv 1 \bmod p$ --- a congruence condition we impose throughout \cref{comparison}.  In this case, it is easy to see that $\Cl(\Q(N^{1/p}))[p]$ --- the $p$-torsion in the class group of $\Q(N^{1/p})$ --- is non-trivial: there is a degree-$p$ subextension $\Q(\zeta_N^{(p)})$ of $\Q(\zeta_N)/\Q$ and $\Q(\zeta_N^{(p)},N^{1/p}) / \Q(N^{1/p})$ is unramified. Letting
\[
r_\Cl = \dim_{\F_p} \Cl(\Q(N^{1/p}))[p]
\]
we can see $r_\Cl \ge 1$, but the exact value of $r_\Cl$ is interesting. In particular, it is interesting to ask when $r_\Cl>1$, or, in other words, when there is an unramified $p$-extension of $\Q(N^{1/p})$ that is not explained by genus theory.

On the modular forms side,  Mazur \cite[Proposition II.9.7]{Mazur77} proved that there is a cuspform $f$ of weight $2$ and level $\Gamma_0(N)$ that is congruent to the Eisenstein series modulo $p$ if and only if $N \equiv 1 \pmod{p}$.  Letting $S_2(\Gamma_0(N);\Z_p)_\mathrm{Eis}$ denote the completion of the space of cusp forms at the Eisenstein maximal ideal, and letting
\[
r_\Eis = \mathrm{rank}_{\Z_p} S_2(\Gamma_0(N);\Z_p)_\mathrm{Eis},
\]
Mazur's result implies $r_\Eis \ge 1$, but he also asked what the significance of $r_\Eis$ is in general \cite[Section II.19, page 140]{Mazur77}.

The first result about $r_\Eis$ was obtained by Mazur \cite[Proposition II.19.2, pg.~140]{Mazur77}, who showed that $r_\Eis=1$ if and only if the Weil pairing on $J_0(N)$ has a certain property. Merel used modular symbols and Mazur's result to prove a remarkable numerical criterion for $r_\Eis$ to equal $1$ \cite[Th\'eor\`eme 2]{Merel96}. More recently, Lecouturier has greatly generalized Merel's techniques to relate the value of $r_\Eis$ to ``higher Merel invariants".

Calegari and Emerton \cite{CE05} were the first to find a relationship between $r_\Eis$ and $r_\Cl$. They proved
\begin{equation}
\label{rank and class group}
r_\Eis=1 \Longrightarrow r_\Cl=1
\end{equation}
using Galois deformation theory and explicit class field theory. Later, Lecouturier \cite{Lecouturier18} used Merel's result to give a new proof of \eqref{rank and class group} by purely algebraic-number-theoretic methods.

The second author and Wang-Erickson refined Calegari and Emerton's method to precisely determine the value of $r_\Eis$ in terms of vanishing of a certain cup product (or, more generally, Massey product) in Galois cohomology \cite{WWE20}. In particular, they show that $r_\Eis=1$ if and only of a certain cup product vanishes \cite[Theorem 1.2.1]{WWE20}. They also show that the vanishing of this cup product implies $r_\Cl>1$, hence giving a new proof of \eqref{rank and class group}.  Schaefer and Stubley \cite{SS19} built upon this cup product technique and the results of \cite{Lecouturier18} to prove more precise bounds on $r_\Cl$.  

\subsection{Comparing $N \equiv 1 \bmod{p}$ and $N \equiv -1 \bmod{p}$}\label{contrasts} When $N \equiv -1 \pmod{p}$, in contrast to the previous section, the genus field of $\Q(N^{1/p})$ is trivial. Hence we see \cref{main thm} as analogous to ``$r_\Cl>1$" of the previous section.

When $N \equiv -1 \pmod{p}$, then Mazur's results imply that $S_2(\Gamma_0(N);\Z_p)_\mathrm{Eis}$ is trivial. Instead, we study $S_2(\Gamma_0(N^2);\Z_p[\zeta_N])_\mathrm{Eis}$ and \cref{main thm 2} implies that this is non-trivial. We think of this as being analogous to ``$r_\Eis>1$" of the previous section.

The surprising thing is that, although ``$r_\Cl>1$" and ``$r_\Eis>1$" do not always hold for $N \equiv 1 \bmod{p}$, their analogs for $N \equiv -1 \pmod{p}$ \emph{do} always hold.  Just as ``$r_\Cl>1$" and ``$r_\Eis>1$" are related to the vanishing of a cup product, their analogs for $N \equiv -1 \bmod{p}$ are also related to the vanishing of a cup product. The difference is that, when $N \equiv -1 \bmod{p}$, the relevant cup product \emph{always vanishes} because the codomain $H^2$ group vanishes.  Indeed, this is the observation that Calegari made after attending a lecture by the second author about the work of \cite{WWE20} explaining the relation between cup products and the class group of $\Q(N^{1/p})$ that allowed him to give a Galois cohomology proof of \cref{main thm} using the methods of \cite{WWE20}.

\subsection{Eisenstein congruences in the case $N \equiv -1 \bmod{p}$}
The purpose of this paper is to show that, just as in \cite{WWE20}, the abstract Galois cochain used in \cite{Frank} actually appears in the Galois representation associated to a newform.  The newform we need has to be congruent to an Eisenstein series, but Mazur's theorem implies that there is no such newform of level $\Gamma_0(N)$ when $N \equiv -1 \bmod{p}$.  Our motivation came from considering the obstruction, from the point of view of Galois deformation theory,  to producing the relevant Galois representation. The observation we made is that there is no obstruction to producing such a representation that is unramified outside $N$ and $p$; the only obstruction comes from making it be Steinberg at $N$.  Consequently,  if we relax the local condition at $N$ by considering forms of level $\Gamma_0(N^2)$,  we expect to find a newform that is congruent to the Eisenstein series.  Although these deformation-theoretic considerations led us to conjecture that \cref{main thm 2} should be true, the proof does not use deformation theory; it is a direct computation using Eisenstein series.

\begin{remark}
In fact, Kobayashi conjectures that for any positive integer $m$ that is divisible by a prime $\ell \equiv -1 \bmod p$, the class number of $\Q(m^{1/p})$ should be divisible by $p$ \cite[Conjuecture 1]{Kobayashi16}.  Calegari explains how to use cup products to prove Kobayashi's conjecture in full generality \cite{Frank}.  We believe a modular approach is also possible by combining the methods of the current paper with those of the second author and Wang-Erickson in \cite{WWE21}, but we have elected not to do so in this paper for simplicity.
\end{remark}

\subsection{Layout}
In \cref{our eisenstein series} we establish the Eisenstein congruence promised in \cref{main thm 2}\ref{congruence}.  There are no Galois representations in this section; the main calculation is to compute the constant terms at the cusps of an Eisenstein series.  We then derive the consequences of this congruence for Galois representations and the class group of $\Q(N^{1/p})$ in \cref{class group}, thus proving \cref{main thm 2}\ref{character} and hence \cref{main thm}.  We end by showing, in \cref{example 5 19}, how explicit information about the Fourier coefficients of the modular form found in \cref{main thm 2} gives explicit information about the primes that split in the class field of $\Q(N^{1/p})$, thus demonstrating the advantages of a modular proof of \cref{main thm}.

\bigskip

\noindent \textit{Notation.}  For a positive integer $n$, let $\zeta_n$ denote a primitive $n$-th root of unity.  When $S$ is a subset of $M_2(\R)$, we write $S^+$ for the subset of $S$ with positive determinant.  For a field $F$ of characteristic 0, fix an algebraic closure $\overline{F}$.  Write $G_F \coloneqq \Gal(\overline{F}/F)$, which we may implicitly view as a subgroup of $G_\Q$ when $F$ is a number field.  Let $G_{\Q,Np}$ denote the Galois group of the maximal extension of $\Q$ that is unramified outside $Np$.  We fix an  embedding $\overline{\Q} \hookrightarrow \overline{\Q}_p$ and let $I_p$ denote the corresponding inertia subgroup of $G_\Q$.  Let $\varepsilon : G_\Q \to \Z_p^\times$ be the $p$-adic cyclotomic character and $\omega$ its mod $p$ reduction.  We write $\overline{\Z}_p$ for the elements in $\overline{\Q}_p$ that are integral over $\Z_p$.  If $X$ is a scheme, then $\OK_X$ denotes its structure sheaf.  

\section{Eisenstein series and residues}\label{our eisenstein series}
\subsection{The modular curve $X_0(N^2)$ and its cusps}
Recall that $N$ and $p$ always denote distinct primes, and $p \geq 5$.  Define
\[
\Gamma \coloneqq \Gamma_0(N^2) \coloneqq \bigl\{\bigl(\begin{smallmatrix}
a & b\\
N^2c & d
\end{smallmatrix}\bigr) \in \SL_2(\Z) \colon c \in \Z\bigr\},
\] 
which acts on the upper half complex plane $\mathfrak{h}$ by M\"obius transformations.  The open Riemann surface $\Gamma \backslash \mathfrak{h}$ can be compactified to $\Gamma \backslash \mathfrak{h}^*$ by adding the cusps $\Gamma \backslash \proj^1(\Q)$.  The complex curves $\Gamma \backslash \mathfrak{h} \subset \Gamma \backslash \mathfrak{h}^*$ descend to $\Q$ and admit a smooth model over $\Z[1/N]$.  Let $Y \coloneqq Y_0(N^2) \subset X \coloneqq X_0(N^2)$ denote the base change of these smooth models to $\Z_p[\zeta_N]$.  

Let $C \coloneqq X \setminus Y$ denote the scheme of cusps on $X$.  A standard calculation shows that there are $N+1$ geometric points of $C$ \cite[\S 3.8]{DS05}, all defined over $\Z_p[\zeta_N]$ \cite[\S VI.5]{DR73}, represented by the following elements in $\proj^1(\Q)=\Q \cup \{\infty\}$:
\begin{equation}\label{cusp identifications}
\infty, 0,1/N, 2/N, \dots, (N-1)/N.
\end{equation}
We sometimes conflate $C$ with its set of geometric points.  It will be convenient to consider $(\Z/N\Z)^\times$ as the indexing set for the set $C \setminus \{\infty, 0\}$.  For $x \in (\Z/N\Z)^\times$, we define $[x] \in C$ to be the class of $\tilde{x}/N \in \proj^1(\Q)$, where $1 \leq \tilde{x} \leq N - 1$ such that $\tilde{x} \equiv x \bmod N$.  Similarly, write $[0]$ for the cusp $0$ to avoid confusion.  We write $\Div(C;\Z_p[\zeta_N])$ for the divisor group supported on the cusps and $\Div^0(C;\Z_p[\zeta_N])$ for the degree-0 part.

\subsection{Modular forms and the residue sequence}
Let $\Omega = \Omega_X^1$ be the invertible sheaf of 1-forms 
on $X$ over $\Z_p[\zeta_N]$.  Viewing $C$ as a divisor on $X$ we have the sheaf $\Omega(C) = \Omega \otimes \OK_X(C)$ of 1-forms on $X$ where we allow simple poles at $C$.  Define the space of modular forms (respectively, cusp forms) of weight 2 and level $\Gamma$ with coefficients in $\Z_p[\zeta_N]$ by $M_2(\Gamma; \Z_p[\zeta_N]) \coloneqq H^0(X, \Omega(C))$ (respectively, $S_2(\Gamma; \Z_p[\zeta_N]) \coloneqq H^0(X, \Omega)$).  Note that this definition is compatible with the usual definition. That is, fixing an isomorphism $\overline{\Q}_p \cong \C$, we can identify $M_2(\Gamma; \Z_p[\zeta_N])$ with the subspace $M_2(\Gamma; \C)$ whose $q$-expansions have $\Z_p[\zeta_N]$-coefficients since $p \nmid N^2$ \cite[Lemma II.4.5]{Mazur77}.  

\begin{proposition}\label{main exact sequence}
There is an exact sequence
\begin{equation}\label{final exact sequence}
0 \to S_2(\Gamma; \Z_p[\zeta_N]) \to M_2(\Gamma; \Z_p[\zeta_N]) \xrightarrow{\Res} \Div(C; \Z_p[\zeta_N])\xrightarrow{\Sigma} \Z_p[\zeta_N] \to 0,
\end{equation}
where $\Res$ sends a modular form to the formal sum of its residues at the cusps, and $\Sigma$ is the sum map.
\end{proposition}

\begin{proof}
Set $W = \Z_p[\zeta_N]$.  For $c \in C$, the inclusion $\iota_c \colon c \hookrightarrow X$ comes with a map $\OK_X \to \iota_{c*}(W)$ with kernel $\OK_X(-c)$.  Putting them together gives a short exact sequence of sheaves
\[
0 \to \OK_X(-C) \to \OK_X \to \bigoplus_{c \in C} \iota_{c*}(W) \to 0
\]
that when tensored with $\Omega(C)$ yields
\[
0 \to \Omega \to \Omega(C) \to \bigoplus_{c \in C} \Omega(C) \otimes \iota_{c*}(W) \to 0.
\]
The associated long exact sequence in cohomology gives the following exact sequence of $W$-modules
\begin{equation}\label{les of sheaves of differentials}
0 \to H^0(X, \Omega) \to H^0(X,\Omega(C)) \to \bigoplus_{c \in C} H^0(X, \Omega(C) \otimes \iota_{c*}(W)) \to H^1(X, \Omega).
\end{equation}

The projection formula and the fact that $\Omega(C)$ is locally free of rank $1$ imply that $H^0(X, \Omega(C) \otimes \iota_{c*}(W)) = W$.  Using the fact that $H^1(X, \Omega)$ is a finitely generated $W$-module together with base change theorems and Serre duality over $W[1/p]$ and $W/p$, we see that $H^1(X, \Omega) = W$.  Tensoring \eqref{les of sheaves of differentials} with $W[1/p]$ and using base change theorems, we see from \cite[Lemma 3.1.13(ii)]{Ohta99} that $H^0(X, \Omega(C)) \to \oplus_{c \in C} W$ is the residue map and the map $\oplus_{c \in C} W \to W$ is just the sum, which is clearly surjective. 
\end{proof}

We briefly recall the formula for the residue map $\Res$ in terms of constant terms of $q$-expansions of modular forms; see \cite[\S 4.5]{Ohta99} for more details.  Given $c \in C(\Z_p[\zeta_N])$, its width is a positive integer $h_c$ such that, up to sign, the stabilizer of $c$ in $\Gamma$ can be conjugated to
\[
\bigl\langle \begin{pmatrix}
1 & h_c\\
0 & 1
\end{pmatrix}\ \bigr\rangle.
\]
In our case, $\infty$ has width 1, $[0]$ has width $N^2$, and all the other cusps of $X$ have width $N$.  The Fourier expansion of $f \in M_2(\Gamma; \Z_p[\zeta_N])$ at $c$ is of the form 
\[
f = \sum_{n = 0}^\infty a_n^c(f)q^{n/h_c},
\]
and 
\[
\Res(f) = \sum_{c \in C(\Z_p[\zeta_N])} h_ca_0^c(f)c.
\]
That is, writing $\Res_c(f)$ for the coefficient of $c$ in $\Res(f)$, we have $\Res_c(f) = h_ca_0^c(f)$.  Note that for the cusp $[0]$, the Atkin-Lehner involution $w_{N^2}$ incorporates the width of $[0]$, and hence $\Res_{[0]}(f) = a_0^{\infty}(w_{N^2}f)$.

\subsection{Hecke operators}
The Hecke operators $T_\ell$ for primes $\ell \nmid N$ and $U_N$ act on $M_2(\Gamma;\Z_p[\zeta_N])$ and this action preserves $S_2(\Gamma;\Z_p[\zeta_N])$. By the sequence \eqref{final exact sequence}, this gives an induced action on $\Div^0(C;\Z_p[\zeta_N])$. In fact, this action extends to $\Div(C;\Z_p[\zeta_N])$ in a way that makes \eqref{final exact sequence} Hecke-equivariant:

\begin{proposition}\label{Hecke formula on cusps}
Define an action of the Hecke operators $T_\ell$ with $\ell \nmid N$ and $U_N$ on $\Div(C;\Z_p[\zeta_N])$ as follows:
\begin{enumerate}
\item For a prime $\ell \neq N$ and $c \in C$ let
\[
T_\ell c = \begin{cases}(\ell+1)c & \mbox{ if $c=0, \infty$} \\
\ell[\ell x] + [\ell^{-1}x] & \mbox{ if $c=[x]$ for $x \in (\Z/N\Z)^\times$}.
\end{cases}
\]
\item For $c \in C$ let
\[
U_N c = \begin{cases} N\cdot [0] & c \neq \infty\\
\infty + \sum_{x \in (\Z/N\Z)^\times}[x] & c = \infty.
\end{cases}
\]
\end{enumerate}
Then the sequence \eqref{final exact sequence} is Hecke-equivariant.
\end{proposition}
\begin{proof}
Note that, to prove the proposition, it suffices to work with $\Q_p$-coefficients (or even $\C$-coefficients), so this is entirely classical.  Note also that this is \emph{not} the ``standard" action of $U_N$, but rather the adjoint action. To explain why this is, we must fix our conventions for Hecke operators.

For any $\alpha \in \GL_2(\Q)^+$, let $\Gamma_\alpha = \Gamma \cap \alpha^{-1}\Gamma\alpha$.  
Then we have two maps $\Gamma_\alpha \backslash \mathfrak{h}^* \to \Gamma \backslash \mathfrak{h}^*$: 
\[
\varphi_\alpha \colon \Gamma_\alpha z \mapsto \Gamma z \text{ and } \psi_\alpha \colon \Gamma_\alpha z \mapsto \Gamma \alpha z.
\]
Define $O_{\alpha} \coloneqq \psi_{\alpha *}\varphi_\alpha^*$ as an operator on all of the cohomology groups in \eqref{les of sheaves of differentials} (base-changed to $\C$).
For any prime $\ell$, let $\alpha_\ell \coloneqq \bigl(\begin{smallmatrix} 1 & 0\\
0 & \ell
\end{smallmatrix}\bigr)$, and let $T_\ell \coloneqq O_{\alpha_\ell}$ and define $U_N \coloneqq T_N$.  With this definition, it is clear that \eqref{final exact sequence} is Hecke-equivariant. To prove the proposition, it remains to see what this action is on $\Div(C)$.

We now consider the standard action of Hecke operators on $\Div(C)$. 
We have the identifications
\[
C = \Gamma \backslash \proj^1(\Q) \isoto \Gamma \backslash \SL_2(\Z)/B(\Z) \isoto \Gamma \backslash \GL_2(\Q)^+ / B(\Q)^+,
\]
where $B \subset \SL_2$ is the upper-triangular Borel.
The inverse of the first map is given by sending the class of $\bigl(\begin{smallmatrix} a & b\\
c & d
\end{smallmatrix} \bigr) \in \SL_2(\Z)$ to $[a \colon c] \in \proj^1(\Q)$, and the second map is induced by the natural inclusion $\SL_2(\Z) \hookrightarrow \GL_2(\Q)^+$.  For an element $\gamma \in \GL_2(\Q)^+$, let $[\gamma] \in C$ denote its class. For $\alpha \in \GL_2(\Q)^+$, the standard action of $O_\alpha$ on $C$ is given by $O_\alpha([\gamma]) = \sum_i [\alpha^{(i)} \gamma]$, where $\Gamma \alpha \Gamma = \coprod_i \Gamma \alpha^{(i)}$. 

The key observation, which we learned from \cite[Proposition 3.4.12]{Ohta99}, is that the identification of $\oplus_{c \in C} H^0(X_{\C}, \Omega(C)_{\C} \otimes \iota_{c*, \C}(\C))$ with $\Div(C)$ swaps standard Hecke operators with their adjoints. (Intuitively, this is because $\oplus_{c \in C} H^0(X_{\C}, \Omega(C)_{\C} \otimes \iota_{c*, \C}(\C))$ is the Serre-dual of $H^0(C,\OK_C)$.) Equivalently, to make \eqref{final exact sequence} Hecke-equivariant,  $T_\ell$ has to act on $\Div(C)$ via the standard action of $O_{\beta_\ell}$, where $\beta_\ell= \bigl(\begin{smallmatrix} \ell & 0\\
0 & 1
\end{smallmatrix}\bigr)$.

Given this, the proposition follows from the following two lemmas, whose 
simple proofs we omit.

\begin{lemma} For a prime $\ell$, let $\beta_\ell= \bigl(\begin{smallmatrix} \ell & 0\\
0 & 1
\end{smallmatrix}\bigr)$.
\begin{enumerate}
\item  Let $\ell \ne N$ be a prime.  A set of representatives for $\Gamma \backslash \Gamma \beta_\ell \Gamma$ is given by $\beta_\ell$ together with any $\ell$ matrices of the form
\[
\beta_\ell\bigl(\begin{smallmatrix}a & b\\
N^2 &d
\end{smallmatrix}\bigr)
\]
where $ad-bN^2=1$ and $d$ ranges over a set of representatives of $\Z/\ell\Z$.
\item A set of representatives for $\Gamma \backslash \Gamma \beta_N \Gamma$ is given by the $N$ matrices
\[
\beta_N\bigl(\begin{smallmatrix}1 &0\\
iN^2 &1
\end{smallmatrix}\bigr)
\]
for $i=0,\dots,N-1$.
\end{enumerate}
\end{lemma}

\begin{lemma}\label{cusps}
For an element $\gamma = \bigl(\begin{smallmatrix}
a & b\\
c & d
\end{smallmatrix} \bigr) \in \GL_2(\Q)^+$, we can determine its class in $C$ as follows. If $c=0$ then the class of $\gamma$ is $\infty$. If $c \ne 0$, write $\frac{a}{c}=\frac{x}{y}$ with $x,y \in \Z$ coprime.
\begin{itemize}
\item If $N^2 \mid y$, then the class of $\gamma$ is $\infty$;
\item if $N \nmid y$, then the class of $\gamma$ is $[0]$;
\item if $y=uN$ with $N \nmid u$, then the class of $\gamma$ is $[ux \bmod{N}]$.
\end{itemize}
\end{lemma}
\end{proof}

\subsection{Eisenstein series}
To prove \cref{main thm 2}\ref{congruence}, we consider congruences between elements of $S_2(\Gamma; \Z_p[\zeta_N])$ and an Eisenstein series with $T_\ell$-eigenvalue $\ell + 1$ for all primes $\ell \neq N$.  There are two such Eisenstein series in $M_2(\Gamma; \Z_p[\zeta_N])$, both old forms.  We write them explicitly.

Define 
\[
E_2(z) \coloneqq \frac{-1}{24} + \sum_{n \geq 1} \sigma(n)q^n, \]
where $\sigma(n) \coloneqq \sum_{0 < d | n} d$.  It is nearly holomorphic of weight 2 and level $1$.  Then 
\[
E_{2, N}(z) \coloneqq E_2(z) - NE_2(Nz)
\]
defines the unique Eisenstein series of weight $2$ and level $\Gamma_0(N)$; its $T_\ell$-eigenvalue is $\ell + 1$ for all primes $\ell \neq N$ and its $U_N$-eigenvalue is $1$.  The constant term of $E_{2, N}$ is $\frac{N-1}{24}$.  Let
\[
E(z) \coloneqq NE_{2, N}(z) - NE_{2,N}(Nz) \in M_2(\Gamma; \Z_p[\zeta_N]),
\]
which has $T_\ell$-eigenvalue $\ell + 1$ for all primes $\ell \neq N$ and $U_N$-eigenvalue $0$.  The constant term of its $q$-expansion at $\infty$ is $0$.

To understand congruences between these Eisenstein series and elements in $S_2(\Gamma; \Z_p[\zeta_N])$, we need to calculate the constant terms of their $q$-expansions at all cusps, not only $\infty$.  To do this, we make use of the Hecke action on the residue sequence \eqref{final exact sequence}.

Let $\T$ be the $\Z_p[\zeta_N]$-subalgebra of $\End_{\Z_p[\zeta_N]}(M_2(\Gamma;\Z_p[\zeta_N]))$ generated by $T_\ell$ for $\ell \nmid N$ and $U_N$, and let $\T' \subset \T$ be the subalgebra generated just by the $T_\ell$.  We consider $\Div^0(C;\Z_p[\zeta_N])$ as a $\T$-module via the action described in \cref{Hecke formula on cusps},  so there is an exact sequence of $\T$-modules
\begin{equation}
\label{eq:res}
0 \to S_2(\Gamma;\Z_p[\zeta_N]) \to M_2(\Gamma;\Z_p[\zeta_N]) \xrightarrow{\Res} \Div^0(C;\Z_p[\zeta_N]) \to 0.
\end{equation}

Let $I$ be the ideal of $\T'$ generated by the elements $T_\ell - \ell - 1$ for $\ell \neq N$ prime and $\m' = (I, p) \subset \T'$, which is maximal.  Write $\cc \coloneqq \sum_{x \in (\Z/N\Z)^\times} ([x] - [0])\in \Div^0(C;\Z_p[\zeta_N])$.  
\begin{proposition}\label{basis for cusps}
When $N \not\equiv 1 \bmod p$, the localization $\Div^0(C;\Z_p[\zeta_N])_{\m'}$  is a free $\Z_p[\zeta_N]$-module of rank 2 with basis
\[
\{\infty - [0]+\cc, \cc\}
\]
that is annihilated by $I$.  Moreover, $U_N$, with $U_N \cc=0$ and $U_N(\infty - [0]+\cc) =\infty - [0]+\cc$. 
\end{proposition}

\begin{proof}
Set $W = \Z_p[\zeta_N]$.  Let $P$ be the $W$-span of $\infty - [0]+\cc$ and $\cc$ in $\Div^0(C;W)$.  The facts that $P$ is annihilated by $I$ and $U_N$ acts as described follow from \cref{Hecke formula on cusps}.  Since $N \not \equiv 1 \bmod p$, the same lemma shows that the section $s : \Div^0(C;W) \to P$ that is the identity on $\infty - [0]$ and sends $[x] - [0]$ to $\frac{1}{N-1}\cc$ for $x \in (\Z/N\Z)^\times$ is $\T'$-equivariant.  Letting $Q = \ker s$, we have a $\T'$-equivariant splitting $\Div^0(C;W)= P \oplus Q$.

To complete the proof, it is enough to show that $\Div^0(C;W/p)[\m'] = P \otimes W/p$. Indeed, this implies that $(Q \otimes W/p)[\m']=0$, from which we deduce that $(Q \otimes W/p)_{\m'}=0$ since $(\T'/p\T')_{\m'}$ is an Artin local ring, and hence $Q_{\m'}=0$.

Suppose $\mathfrak{a}=a_\infty(\infty - [0]) + \sum_{x \in (\Z/N\Z)^\times} a_x([x] - [0]) \in \Div^0(C;W/p)$ is annihilated by $\m'$. 
Equivalently, $\mathfrak{a}$ is annihilated by $t_\ell \coloneqq T_\ell-\ell-1$ for all primes $\ell \ne N$. We will use the formulas from \cref{Hecke formula on cusps} to show this implies that $a_x=a_1$ for all $x \in (\Z/N\Z)^\times$. Hence we will have
\[
\mathfrak{a} = a_\infty (\infty-[0]) + a_1 \cc \in P \otimes_W W/p.
\]

To show that $a_x=a_1$ for all $x \in (\Z/N\Z)^\times$, choose a prime $\ell$ such that $\ell \equiv -1 \pmod{p}$ and such that $\ell$ is a primitive root modulo $N$. Then, since $t_\ell \mathfrak{a}=0$, we have $a_{\ell x} = a_{\ell^{-1}x}$ for all $x$. Since $\ell$ is a primitive root, this implies that $a_x=a_1$ if $x$ is a square, and $a_x=a_{\ell}$ if $x$ is a non-square. Now take $q \not \equiv -1 \pmod{p}$ to be a prime that is a not a square modulo $N$. Then, since $t_q \mathfrak{a}=0$, we have
\[
(q+1)a_1 = qa_{q^{-1}} + a_q.
\]
Since $q$ is not a square, $a_q=a_{q^{-1}}=a_\ell$, so we have 
\[
(q+1)a_1=(q+1)a_\ell
\]
which implies $a_1=a_\ell$ because $q \not\equiv -1 \pmod{p}$. Hence $a_x=a_1$ for all $x \in (\Z/N\Z)^\times$, so $\mathfrak{a}$ is in $P \otimes W/p$. 
\end{proof}

\begin{corollary}\label{ResE}
We have $\Res(E) = \frac{N^2-1}{24} \cc$ and $\Res(E_{2,N}) = \frac{N-1}{24}(\infty -[0] +\cc)$. 
\end{corollary}

\begin{proof}
Since $\Res$ is Hecke equivariant, it follows that $\Res(E) = \alpha \cc$, and $\Res(E_{2,N})=\beta(\infty-[0]+\cc)$, where $\alpha = \frac{1}{N-1}\Res_{[0]} E$ and $\beta=\Res_\infty(E_{2,N})=a_0(E_{2,N})$.  This proves the claim for $E_{2,N}$. Letting $w_{N^2}$ be the Atkin-Lehner operator, we can calculate $\Res_{[0]} E = a_0(w_{N^2}E)$.  As $(w_{N^2}E)(z) = E_{2, N} - N^2E_{2, N}(Nz)$, we see that $a_0(w_{N^2}E) = (1-N^2)a_0(E_{2,N}) = \frac{(1 - N^2)(N-1)}{24}$ and hence $\alpha = \frac{N^2-1}{24}$.
\end{proof}

Let $\m \subset \T$ be the maximal ideal generated by $\m'$ and $U_N$, and let $\m_\mathrm{old} \subset \T$ be the maximal ideal generated by $\m'$ and $U_N-1$. We consider the localizations at these two maximal ideals.

\begin{lemma}
\label{lem:localization and congruence} \hfill 
\begin{enumerate}
\item We have $S_2(\Gamma;\Z_p[\zeta_N])_\m \ne 0$ if and only if there is a form $f \in S_2(\Gamma;\Z_p[\zeta_N])$ such that $a_n(f) \equiv a_n(E) \bmod{p}$ for all $n \ge 1$.
\item We have $S_2(\Gamma;\Z_p[\zeta_N])_{\m_\mathrm{old}} \ne 0$ if and only if there is a form $f \in S_2(\Gamma;\Z_p[\zeta_N])$ such that $a_n(f) \equiv a_n(E_{2,N}) \bmod{p}$ for all $n \ge 1$.
\end{enumerate}
\end{lemma}
\begin{proof}
We prove (1), the proof of (2) being identical.  If such a form $f$ exisits, it clearly gives a non-zero element of $S_2(\Gamma;\Z_p[\zeta_N])_\m$. Conversely, suppose that $S_2(\Gamma; \Z_p[\zeta_N])_{\m} \ne 0$. Then, reducing modulo $p$, we have $S_2(\Gamma; \F_p(\zeta_N))_{\m} \ne 0$. Since $\T_{\m}/p \T_{\m}$ is Artinian, this implies that there is a non-zero element $\bar{f}$ of $S_2(\Gamma; \F_p(\zeta_N))[{\m}]$. Since $\bar{f}$ is annihilated by $\m$, we have  $a_n(\bar{f})=a_n(E) \bmod{p}$ for all $n$.  We can now let $f \in S_2(\Gamma; \Z_p[\zeta_N])$ be any lift of $\bar{f}$.
\end{proof}

\begin{theorem}
\label{thm:existance of congruence}
Assume $N \not\equiv 1 \bmod p$.  Then $S_2(\Gamma; \Z_p[\zeta_N])_{\m_\mathrm{old}}=0$ and there is a $\T_{\m}$-equivariant short exact sequence
\begin{equation}
\label{eq:res at m}
0 \to S_2(\Gamma; \Z_p[\zeta_N])_{\m} \to M_2(\Gamma; \Z_p[\zeta_N])_{\m} \xrightarrow{\Res} \Z_p[\zeta_N]\cdot \cc\to 0.
\end{equation}
Moreover, $S_2(\Gamma;\Z_p[\zeta_N])_\m=0$ if and only if $N \not\equiv -1 \bmod p$. 

In particular, there exists $f = \sum_{n \geq 1} a_nq^n \in S_2(\Gamma; \Z_p[\zeta_N])$ such that $a_\ell \equiv \ell + 1 \bmod p$ for all primes $\ell \neq N$ if and only if $N \equiv -1 \bmod p$.
\end{theorem}

\begin{proof}
Suppose that $S_2(\Gamma; \Z_p[\zeta_N])_{\m_\mathrm{old}} \ne 0$.  Then, by \cref{lem:localization and congruence}, there is an $f \in S_2(\Gamma;\Z_p[\zeta_N])$ with $a_n(f) \equiv a_n(E_{2,N}) \bmod{p}$ for all $n \ge 1$. Since $a_0(E_{2,N}) \not\equiv 0 \bmod{p}$, this implies that there is a non-zero constant in $M_2(\Gamma; \F_p(\zeta_N))$, a contradiction.

The exact sequence \eqref{eq:res at m} follows directly from \eqref{eq:res} and \cref{basis for cusps}.  If $N \equiv -1 \pmod{p}$, then let $g \in M_2(\Gamma; \Z_p[\zeta_N])_{\m}$ be such that $\Res(g)=\cc$ and let $f=E-\frac{N^2-1}{24}g$. Then $f \equiv E \pmod{p}$ and, since $\Res(f)=0$, we have $f \in S_2(\Gamma; \Z_p[\zeta_N])_{\m}$.

Conversely,  suppose that $N \not\equiv \pm 1 \pmod{p}$ and, for the sake of contradiction, that $S_2(\Gamma;\Z_p[\zeta_N])_\m \ne 0$.  Let $\bar{E} \in M_2(\Gamma;\F_p(\zeta_N))$ be the reduction of $E$ modulo $p$.  By \cref{lem:localization and congruence}, there is an $f \in S_2(\Gamma;\F_p(\zeta_N))$ with $a_n(f) = a_n(\bar{E})$ for all $n \ge 1$.
Since also $a_0(\bar{E})=a_0(f)=0$, this implies that $f=\bar{E}$ by the $q$-expansion principle. This implies $\bar{E} \in S_2(\Gamma;\F_p(\zeta_N))$, but $\Res(\bar{E}) \ne 0$ by \cref{ResE}, a contradiction.

For the final statement, simply note that such an $f$ must belong to $S_2(\Gamma;\Z_p[\zeta_N])_{\m'}$ and that
\[
S_2(\Gamma;\Z_p[\zeta_N])_{\m'} = S_2(\Gamma;\Z_p[\zeta_N])_{\m_\mathrm{old}}  \oplus S_2(\Gamma;\Z_p[\zeta_N])_{\m} 
\]
since $\m_\mathrm{old}$ and $\m$ are the only maximal ideals of $\T$ containing $\m'$.
\end{proof}

Let $\T_\m^0$ be the maximal quotient of $\T_\m$ acting faithfully on $S_2(\Gamma; \Z_p[\zeta_N])_\m$.  Recall that by duality, minimal prime ideals $\Pp$ of $\T_\m^0$ are in one-to-one correspondence with normalized eigenforms in $S_2(\Gamma; \overline{\Z}_p)$ that are congruent to $E$ modulo the unique prime above $p$ in the $p$-adic ring $\T_\m^0/\Pp$.  By \cref{thm:existance of congruence}, we know that $\T_\m^0 \neq 0$ when $N \equiv -1 \bmod p$.  Moreover, we know that the eigenform corresponding to any minimal prime must be a newform because, by Mazur's theorem, there are no oldforms that are congruent to $E$.
Thus we have the following corollary, which gives \cref{main thm 2}\ref{congruence}.

\begin{corollary}\label{eigenform congruence}
Assume that $N \equiv -1 \bmod p$. Then there is a newform $f$ of weight $2$ and level $\Gamma_0(N^2)$ and a prime ideal $\p$ over $p$ in the ring of integers $\mathcal{O}_f$ of the Hecke field of $f$ such that $a_\ell(f) \equiv 1+\ell \bmod \p$ for all primes $\ell$.
\end{corollary}

\section{An unramified $p$-extension of $\Q(N^{1/p})$ when $N \equiv -1 \bmod p$}\label{class group}
In this section we use a congruence between a cusp form and the Eisenstein series $E$ from \cref{eigenform congruence} to give a modular construction of a degree $p$ unramified extension of $\Q(N^{1/p})$ when $N \equiv -1 \bmod p$, thus proving \cref{main thm} and the second half of \cref{main thm 2}.  Throughout this section we assume $N \equiv -1 \bmod p$.  

As in \cref{eigenform congruence}, fix an eigenform $f \in S_2(\Gamma; \overline{\Z}_p)$ that is congruent to $E$ modulo the prime $\p$ lying over $p$.  
Let $\Q_p(f)/\Q_p$ be the field generated by the Hecke eigenvalues of $f$, $\OK$ its ring of integers, $\varpi$ a uniformizer, and $\F = \OK/\varpi$.  Let $s \geq 1$ be the largest integer such that $f \equiv E \bmod \varpi^s$, so $s$ is the largest integer such that $a_\ell(f) \equiv 1+\ell \bmod{\varpi^s}$ for all primes $\ell \neq N$.

We now recall a Galois-theoretic interpretation of the integer $s$. Let $\rho_f \colon G_{\Q,Np} \to \GL(V_f) \cong \GL_2(\Q_p(f))$ the Galois representation corresponding to $f$ and let $t_f = \mathrm{tr}(\rho_f): G_{\Q,Np} \to \OK$ be its trace. Recall the \emph{reducibility ideal} of $t_f$, as defined in \cite[Section 1.5]{BC2009}: it is the smallest ideal $J \subset \OK$ such that $t_f \equiv \psi_1 + \psi_2 \bmod{J}$ for characters $\psi_i: G_{\Q} \to (\OK/J)^\times$. 

\begin{lemma}
\label{lem:red ideal}
The reducibility ideal of $t_f$ is $\varpi^s\OK$.
\end{lemma}
\begin{proof}
Let $J \subset \OK$ be the reducibility ideal of $t_f$ and write $t_f \equiv \psi_1 + \psi_2 \bmod{J}$ for characters $\psi_i: G_{\Q} \to (\OK/J)^\times$.  Since $\det(\rho_f)=\epsilon$, we can write $\psi_1=\psi\epsilon$ and $\psi_2=\psi^{-1}$ for a character $\psi: G_{\Q,Np} \to (\OK/J)^\times$ with $\psi \equiv 1 \bmod{\varpi\OK}$.  In particular, $\psi$ has $p$-power order. We claim that $\psi$ is trivial. Assuming this claim, we see that $J=\varpi^t\OK$ for the largest integer $t$ such that $t_f \equiv \epsilon+1 \bmod{\varpi^t}$, so $t=s$ by Cheboterov density. 

It remains to show that $\psi$ is trivial. For this, note that $f$ is ordinary, since $a_p(f) \equiv a_p(E) \equiv 1 \bmod{\varpi\OK}$.  Hence we have $t_f|I_p = \epsilon +1$, so we see that $\psi$ is unramified at $p$. Then $\psi$ factors through the maximal unramified-outside-$N$ abelian pro-$p$ extension of $\Q$, which is trivial since $N \not \equiv 1 \bmod{p}$. Hence $\psi$ is trivial. 
\end{proof}

 The goal of this section is to prove \cref{main thm 2}\ref{character}, which in turn implies \cref{main thm}.  In particular, we show the following.

\begin{theorem}\label{finding chi}
There is an everywhere unramified order-$p$ character $\chi : G_{\Q(N^{1/p})} \to (\OK/\varpi^{s+1})^\times$ such that $t_f|G_{\Q(N^{1/p})} \equiv \chi\epsilon + \chi^{-1} \bmod \varpi^{s+1}$.
\end{theorem}

\subsection{Constructing an unramified $p$-extension}
For a $G_\Q$-stable $\OK$-lattice $T \subset V_f$, let $\rho_T \colon G_{\Q,Np} \to \GL(T)$ denote the corresponding representation.

\begin{lemma}\label{lem:rhoT shape}
There is a $G_\Q$-stable lattice $T \subset V_f$ such that
\[
\rho_T \bmod{\varpi^s} = \begin{pmatrix} \epsilon & \kappa_N\\
0 & 1
\end{pmatrix},
\]
where $\kappa_N \colon G_\Q \to (\OK/\varpi^s)(1)$ is the Kummer cocycle accociated to $N$. Moreover, if we write $\rho_T=\sm{a}{b}{c}{d}$, then the ideal generated by $c(\sigma)$ for all $\sigma \in G_\Q$ is $\varpi^s\OK$.
\end{lemma}
\begin{proof}
Using Ribet's lemma \cite[Proposition 2.1]{Ribet76}, we choose $T$ such that
\[
\rho_T \bmod{\varpi}= \begin{pmatrix} \omega & \bar{b}\\
0 & 1
\end{pmatrix}
\]
where $\bar{b} \colon G_\Q \to \F(1)$ is a cocycle with non-trivial cohomology class.  Write the entries of $\rho_T$ as
$\rho_T = \sm{a}{b}{c}{d}$.
By \cref{lem:red ideal} and \cite[Proposition 1.5.1, pg.~35]{BC2009}, we see that $BC = \varpi^s\OK$, where $B, C \subset \OK$ are the ideals generated by $b(\sigma)$ and $c(\sigma)$, respectively,  for all $\sigma \in G_{\Q,Np}$.  Since $\bar{b}$ is non-trivial, we see that $B$ is the unit ideal, so $C= \varpi^s\OK$.  From this we see that $\rho_T \bmod \varpi^s$ has the desired upper-triangular shape, and it remains to describe the cocycle $b \bmod{\varpi^s}$.

The class of $b \bmod{\varpi^s}$ belongs to 
\[H^1(G_{\Q,Np}, (\OK/\varpi^s)(1))
\]
which is generated by the Kummer classes $\kappa_N$ and $\kappa_p$ of $N$ and $p$ by Kummer theory.  However, the fact that $\rho_T$ is \emph{finie} in the sense of Serre \cite{Serre1987} (that is, it comes from the generic fiber of a finite flat group scheme over $\Z_p$) forces $b \bmod{\varpi^s}$ to lie in the subgroup genertated by $\kappa_N$ (see \cite[Section 2]{Serre1987}). Since $\bar{b}$ is non-trivial, we see that $b \bmod{\varpi^s}$ must be a unit multiple of $\kappa_N$, and we can change basis to ensure that $\rho_T \bmod{\varpi^s}$ has the desired form.
\end{proof}

Fix a lattice $T \subset V_f$ as in the lemma, and let $\rho=\rho_T = \sm{a}{b}{c}{d}$ and $\bar{b} = b \bmod{\varpi}$.  Since $\bar{b}$ is the Kummer cocycle associated to $N$, we see that $\bar{b}|G_F=0$, where $F \coloneqq \Q(N^{1/p})$. This implies that $b|G_F$ takes values in $\varpi \OK$. Since we also know that $c$ takes values in $\varpi^s \OK$, we see that the reducibility ideal (in the sense of \cite[Section 1.5]{BC2009}) of $t_f |G_F$ is contained in $\varpi^{s+1}$. This implies that
\[
a|G_F \bmod{\varpi^{s+1}}: G_F \to (\OK/\varpi^{s+1})^\times
\]
is a group homomorphism. Define a character $\chi: G_F\to (\OK/\varpi^{s+1})^\times$ by
\[
\chi \epsilon = a|G_F \bmod{\varpi^{s+1}}.
\]
Note that, since $\det(\rho)=\epsilon$, we have $d|G_F \bmod{\varpi^{s+1}} = \chi^{-1}$, so
\begin{equation}
\label{eq:t_f on G_F}
t_f|G_F=\chi\epsilon + \chi^{-1} \bmod{\varpi^{s+1}}
\end{equation}
To complete the proof of \cref{finding chi}, it suffices to prove that $\chi$ has order $p$ and is unramified everywhere, which we prove in the following two propositions.

\begin{proposition}
The character $\chi$ is non-trivial.
\end{proposition}
\begin{proof}
Let $r = \rho \bmod{\varpi^{s+1}}$. Suppose, for the sake of contradiction, that $\chi$ is trivial. Then we have
\[
r(G_{F(\zeta_p)})  \subset \left \{ \ttmat{x}{\varpi y}{\varpi^s z}{1} \in \GL_2(\OK/\varpi^{s+1}) \right \}. 
\]
Now choose $\sigma \in G_\Q$ such that $\bar{b}(\sigma) \ne 0$, and let $\tau \in G_{F(\zeta_p)}$ and write $r(\tau) = \sm{x}{\varpi y}{\varpi^s z}{1}$.
Since $x, d(\sigma) \equiv 1 \bmod \varpi$, we compute that
\[
r(\sigma) \ttmat{x}{\varpi y}{\varpi^s z}{1} r(\sigma)^{-1} = \ttmat{*}{*}{*}{1-zb(\sigma)\det(r(\sigma))^{-1}\varpi^s}.
\]
Since $r(G_{F(\zeta_p)})$ is normal in $r(G_\Q)$, we conclude that $zb(\sigma)\det(r(\sigma))^{-1}\varpi^s=0$. Since $b(\sigma)$ and $\det(r(\sigma))$ are units, we see that $z \in \varpi\OK/\varpi^{s+1}$. This shows that, for any $\tau \in G_{F(\zeta_p)}$, we have $c(\tau) \equiv 0 \bmod{\varpi^{s+1}}$.

By \cref{lem:rhoT shape}, we can write $c = \varpi^s \tilde{c}$ for a cochain $\tilde{c}: G_{\Q,Np} \to \OK$ such that $\bar{c} \coloneqq \tilde{c} \bmod{\varpi}$ is non-trivial.  It follows that $\bar{c}$ is a cocycle $\bar{c}:G_\Q \to \F(-1)$ with non-trivial class.  Then $\bar{c}|G_{\Q(\zeta_p)}$ is a homomorphism cutting out a degree-$p$ extension $K/\Q(\zeta_p)$ such that $\Gal(\Q(\zeta_p)/\Q)$ acts on $\Gal(K/\Q(\zeta_p))$ via $\omega^{-1}$.  But, since we assume that $\chi$ is trivial, the previous paragraph shows that $\bar{c}(\tau)=0$ for all $\tau \in G_{F(\zeta_p)}$, so $K=F(\zeta_p)$. But $\Gal(\Q(\zeta_p)/\Q)$ acts on $\Gal(F(\zeta_p)/\Q(\zeta_p))$ via $\omega$,  so this implies $\omega=\omega^{-1}$ which is a contradiction because we assume $p>3$. 
\end{proof}

\begin{proposition}\label{L/F unramified}
The character $\chi$ is unramified everywhere.
\end{proposition}
\begin{proof}

Since $\rho$ is unramifed outside $Np$, $\chi$ is as well. It remains to show that $\chi$ is unramified at $N$ and $p$. We first consider ramification at $N$. By local class field theory, the maximal abelian tame quotient of the inertia group at $N$ in $G_F$ has order $N-1$, which is prime-to-$p$ by our assumptions that $N \equiv -1 \bmod{p}$ and $p>2$. Since the image of $\chi$ has order $p$, this implies that $\chi$ is unramified at $N$.

Finally, to see that $\chi$ is unramified at $p$ we need to show that $\chi|I_p \cap G_F = 0$.  By \cref{lem:rhoT shape} and the definition of $\chi$, we can write $\chi=1+\alpha \varpi^s$ for an additive character $\alpha: G_F \to \F$, and we need to show that $\alpha|I_p\cap G_F=0$.  In this notation, \eqref{eq:t_f on G_F} says
\[
t_f|G_F \equiv \varepsilon + 1 + (\varepsilon - 1)\alpha \varpi^s \bmod \varpi^{s+1}
\]
On the other hand, as we noted in the proof of \cref{lem:red ideal}, the fact that $f$ is ordinary implies that $t_f|I_p=\epsilon +1$.  Combining these two,  we have $(\varepsilon - 1)\alpha\varpi^s = 0$ on $I_p \cap G_F$.  This is equivalent to $(\omega - 1)\alpha = 0$ as functions $I_p \cap G_F \to \F$.  If $\sigma \in I_p \cap G_F \setminus \ker \omega$, then it follows that $\alpha(\sigma) = 0$.  For $\sigma \in I_p \cap G_F \cap \ker \omega$, choose any $\tau \in I_p \cap G_F \setminus \ker \omega$.  Then $\omega(\tau) \neq 1$ and $\alpha(\tau) = 0$, so we obtain
\[
0 = (\omega(\sigma\tau)-1)\alpha(\sigma\tau) = (\omega(\sigma)\omega(\tau)-1)(\alpha(\sigma)+\alpha(\tau)) =( \omega(\tau)-1)\alpha(\sigma),
\]
and thus $\alpha(\sigma)=0$.
\end{proof}

\subsection{Explicit class field theory encoded by $f$}
\label{example 5 19}
Keep the notation from the previous section.  In particular, $N, p, f$ are all fixed.  From $f$, there is an associated character $\chi$ of $G_{\Q(N^{1/p})}$ as in \cref{main thm 2}\ref{character}, or equivalently as defined prior to \cref{L/F unramified}.  Let $L$ be the degree $p$ extension of $\Q(N^{1/p})$ cut out by character $\chi$, so $L/\Q(N^{1/p})$ is an everywhere unramified degree $p$ extension.  In this section we show how the Fourier coefficients of $f$ (modulo $\varpi^{s+1}$) carry information about how primes of $\Q(N^{1/p})$ split in $L$ as well as when $N$ is not a $p$-th power modulo $\ell$ when $\ell \equiv 1 \bmod p$.  This explicit information shows the advantage of our modular methods compared to Calegari's abstract cup product argument discussed at the end of \cref{contrasts}.

We begin by understanding how rational primes split in $\Q(N^{1/p})$.  For each prime $\ell$, fix a discrete logarithm $\log_\ell : \F_\ell^\times \to \Z/(\ell - 1)\Z$, so $a \in \F_\ell^\times$ is a $p$-th power if and only if $\log_\ell(a) \equiv 0 \bmod p$, which is automatic whenever $\ell \not\equiv 1 \bmod p$.

\begin{lemma}\label{splitting in F}
Let $\ell \neq p, N$ be prime and let $r$ be the multiplicative order of $\ell$ in $\F_p^\times$. If $\log_\ell(N) \not\equiv 0 \bmod p$, then $\ell$ is inert in $\Q(N^{1/p})$.  Otherwise,  there are $\frac{p-1}{r}+1$ primes of $\Q(N^{1/p})$ lying over $\ell$, one with residue degree $1$ and the rest having residue degree $r$.
\end{lemma}

\begin{proof}
Note that the only prime factors dividing the discriminant of the order $\Z[N^{1/p}]$ are $p$ and $N$ --- the same prime divisors of the discriminant of $\Q(N^{1/p})$.  Thus we can understand the splitting behavior of $\ell$ in $\Q(N^{1/p})$ by considering how $x^p-N$ factors over $\F_\ell$.

First suppose $\log_\ell(N) \not\equiv 0 \bmod p$, so $N$ is not a $p$-th power in $\F_\ell$. Then it is well-known that $x^p-N$ is irreducible over $\F_\ell$ (see \cite[Theorem VI.9.1, pg.~297]{Lang2002}, for example). 

Now suppose $\log_\ell(N) \equiv 0 \bmod p$, so $N = a^p$ for some $a \in \F_\ell^\times$.  Using the substitution $x \mapsto ay$, we get
\[
\frac{\F_\ell[x]}{(x^p-N)} \cong \frac{\F_\ell[y]}{(y^p - 1)} \cong \F_\ell \times \frac{\F_\ell[y]}{(\Phi_p(y))} \cong \F_\ell \times \F_{\ell^r}^{(p-1)/r},
\]
where $\Phi_p(y)$ is the cyclotomic polynomial. For the last isomorphism, note that $\Phi_p(y)$ divides $y^{\ell^r}-y$ but $\gcd(\Phi_p(y),y^{\ell^d}-y)=1$ for any $d<r$,  so the irreducible factors of $\Phi_p(y)$ all have degree $r$. 
\end{proof}

\begin{proposition}\label{explicit what we know}
Let $\ell \neq N, p$ be prime.
\begin{enumerate}
\item If $\ell \equiv 1 \bmod p$ and $a_\ell(f) \not\equiv 2 \bmod \varpi^{s+1}$, then $\log_\ell(N) \not\equiv 0 \bmod p$, so $\ell$ is inert in $\Q(N^{1/p})$ and $\ell\OK_{\Q(N^{1/p})}$ splits completely in $L$ (in fact, in the Hilbert class field of $\Q(N^{1/p})$).
\item If $\ell \not\equiv 1 \bmod p$, then the unique prime of $\Q(N^{1/p})$ lying over $\ell$ of residue degree $1$ splits in $L$ if and only if $a_\ell(f) \equiv \ell+1 \bmod \varpi^{s+1}$.
\end{enumerate}
\end{proposition}

\begin{proof}
Write the character $\chi$ from \cref{main thm 2}\ref{character} as $\chi = 1 + \varpi^s\alpha$ with $\alpha : G_{\Q(N^{1/p})} \to \F$ an additive character.  Then $t_f|G_{\Q(N^{1/p})} = \omega\chi + \chi^{-1} = 1 + \omega + \varpi^s(\omega-1)\alpha$.  Suppose that $\log_\ell(N) \equiv 0 \bmod p$ so that $\ell$ has a prime $\lambda$ of $\Q(N^{1/p})$ lying above it of residue degree $1$.  Up to conjugation we may take $\Frob_\ell = \Frob_\lambda \in G_{\Q(N^{1/p})}$.  Thus by \eqref{eq:t_f on G_F}, we have
\begin{equation}\label{mod s+1}
a_\ell(f) = t_f(\Frob_\lambda) \equiv 1 + \ell + \varpi^s(\ell -1)\alpha(\Frob_\lambda) \bmod \varpi^{s+1}
\end{equation}
whenever $\log_\ell(N) \equiv 0 \bmod{p}$. 

When $\ell \not\equiv 1\bmod p$, we see that $a_\ell(f) \equiv \ell + 1 \bmod \varpi^{s+1}$ if and only if $\alpha(\Frob_\lambda) = 0$.  Since $L$ is cut out by $\ker \chi = \ker \alpha$, it follows that $\lambda$ splits in $L$ if and only if $a_\ell(f) \equiv \ell+1 \bmod \varpi^{s+1}$, proving (2).  

In contrast, when $\ell \equiv 1 \bmod p$, \eqref{mod s+1} shows that $a_\ell(f) \equiv 2 \bmod \varpi^{s+1}$ under the assumption $\log_\ell(N) \equiv 0 \bmod p$, thus establishing the contrapositive of (1).  The last part of (1) follows from \cref{splitting in F} and the fact that the principal ideals of $\Q(N^{1/p})$ are exactly those that split completely in its Hilbert class, which contains $L$.  
\end{proof}

\begin{remark}
Note that if a prime $\ell$ satisfies $a_\ell(f) \equiv \ell + 1 \bmod \varpi^{s+1}$, then $T_\ell$ cannot generate the Eisenstein ideal since that would force the entire Eisenstein congruence to persist modulo $\varpi^{s+1}$, contradicting the definition of $s$. 
\end{remark}

If we impose the hypothesis that the class number of $\Q(N^{1/p})$ is $p$, so $L$ is its Hilbert class field, then we can further interpret our results in a classical style suggestive of results in explicit class field theory in the case of imaginary quadratic fields.  While this hypothesis on the class number is certainly not always satisfied, it holds in many examples.  For instance, the hypothesis holds when $p = 5$ and $N \in \{19, 29, 59, 79, 89, 109, 139, 149, 199\}$ and when $p = 7$ and $N \in \{13, 41, 97, 139, 181\}$.

\begin{corollary}\label{add class number assumption}
Assume that $\Q(N^{1/p})$ has class number $p$.  Let $\lambda$ be a prime of $\Q(N^{1/p})$ lying over $\ell \neq N, p$.
\begin{enumerate}
\item\label{equivalent conditions} Suppose that $\ell \not \equiv 1 \bmod p$ and either $\#\OK_{\Q(N^{1/p})}/\lambda = \ell$ or $\#\OK_{\Q(N^{1/p})}/\lambda = \ell^{p-1}$.  Then the following are equivalent:
\begin{enumerate}
\item $\lambda$ is a principal $\OK_{\Q(N^{1/p})}$-ideal;
\item $\lambda$ splits completely in $L$ over $\Q(N^{1/p})$;
\item $\ell$ is a norm from $\Q(N^{1/p})$;
\item $a_\ell(f) \equiv \ell + 1 \bmod \varpi^{s+1}$.
\end{enumerate}
\item If $\ell \equiv 1 \bmod p$ and $a_\ell(f) \not\equiv 2 \bmod \varpi^{s+1}$, then $\ell$ is inert in $\Q(N^{1/p})$ and then splits in $L$.  In this case $\ell$ is not a norm from $\Q(N^{1/p})$.
\end{enumerate}
\end{corollary}

\begin{proof}
Write $F \coloneqq \Q(N^{1/p})$, and suppose that $\ell \not \equiv 1 \bmod p$ and $\#\OK_F/\lambda = \ell$.  The equivalence of (a) and (b) follows from the fact that the primes that split in the Hilbert class field $L$ of $F$ are exactly the principal ideals.  The equivalence of (a) and (c) follows from the fact that the norm of an element is equal to the norm of the ideal it generates.  The equivalence of (b) and (d) follows from the first part of \cref{explicit what we know}.

The second part follows from \cref{explicit what we know} and the fact that $L$ is the Hilbert class field of $F$ by assumption.
\end{proof}

\begin{example}
We finish with an example when $p=5$ and $N = 19$.  We compute that $\Q(19^{1/5})$ has class number $5$, so \cref{add class number assumption} applies. In this case $f$ has LMFDB label 361.2.a.f and Hecke field $\Q(\sqrt{5})$.  The Eisenstein congruence holds modulo $\varpi = \sqrt{5}$, but not modulo $\varpi^2 = 5$, so $s = 1$ in this case.  Set $\beta = \frac{1 + \sqrt{5}}{2}$.  \cref{table} contains the first sixty prime-index coefficients for $f$.  The ones in bold are those for which the Eisenstein congruence persists modulo $5$, and the circled primes $\ell$ are those for which \cref{add class number assumption} implies that there exists a principal prime ideal of $F$ lying over $\ell$.  (We also circle $19$ since the principal ideal generated by $19^{1/5}$ clearly lies over it.)  Moreover, in this example we can calculate that $\OK_F = \Z[19^{1/5}]$ and hence it is easy to write the norm form explicitly.  In particular, the four equivalent conditions on $\ell$ in \cref{add class number assumption}\eqref{equivalent conditions} are also equivalent to 
\begin{align*}
\ell  = a^{5} &	-95 a^{3} b e -95 a^{3} c d + 95 a^{2} b^{2} d + 95 a^{2} b c^{2} + 1805 a^{2} c e^{2} + 1805 a^{2} d^{2} e -95 a b^{3} c + 1805 a b^{2} e^{2}  \\
& -1805 a b c d e -1805 a b d^{3} -1805 a c^{3} e + 1805 a c^{2} d^{2} -34295 a d e^{3} + 19 b^{5} -1805 b^{3} d e \\
&  + 1805 b^{2} c^{2} e + 1805 b^{2} c d^{2} -1805 b c^{3} d -34295 b c e^{3} + 34295 b d^{2} e^{2} + 361 c^{5} + 34295 c^{2} d e^{2} \\
& -34295 c d^{3} e + 6859 d^{5} + 130321 e^{5}
\end{align*}
for some $(a,b,c,d,e) \in \Z^5$.

\begin{table}
\centering
\begin{tabular}{| c | c || c | c || c | c || c | c | }
\hline
$\ell$ & $a_\ell(f)$ & $\ell$ & $a_\ell(f)$ & $\ell$ & $a_\ell(f)$ & $\ell$ & $a_\ell(f)$\\
\hline\hline
2 & $\beta$ & 53 & $5-7\beta$ & 127 & $9-2\beta$ & 199 & $6-12\beta$\\
3 & $2-\beta$ & 59 & $-11 + 7\beta$ & 131 & \boldsymbol{$7+5\beta$} & \Circle{211} & $1-3\beta$\\
5 & $2\beta$ & \Circle{61} & $-7-2\beta$ & 137 & $1 + 4\beta$ & 223 & $11-14\beta$ \\
\Circle{7} & \boldsymbol{$3$} & \Circle{67} & \boldsymbol{$-7$} & 139 & $-3+11\beta$ & 227 & $-3+12\beta$\\
\Circle{11} & $-\beta$ & \Circle{71} & $-1-4\beta$ & \Circle{149} & \boldsymbol{$-10+5\beta$} & 229 & $-12-\beta$\\
\Circle{13} & \boldsymbol{$-1$} & 73 & $7-6\beta$ & 151 & \boldsymbol{$-13 - 5\beta$} & 233 & $11-4\beta$\\
17 & $4 -2 \beta$ & 79 & $-6+12\beta$ & 157 & $-13-3\beta$ & 239 & $11-7\beta$\\
\Circle{19} & \boldsymbol{$0$} & 83 & $2+4\beta$ & 163 & $5-2\beta$ & 241 & \boldsymbol{$-13+10\beta$}\\
23 & $7 - \beta$ & 89 & $-11+2\beta$ & 167 & $17+2\beta$ & 251 & \boldsymbol{$7-20\beta$}\\
29 & $-2 - \beta$ & 97 & $9+3\beta$ & 173 & $6-4\beta$ & \Circle{257} & \boldsymbol{$-12 + 20\beta$}\\
\Circle{31} & $-4-3\beta$ & 101 & \boldsymbol{$7-10\beta$} & 179 & $9+2\beta$ & 263 & $-2 - 8\beta$\\
37 & $4 + 3\beta$ & 103 & $3+7\beta$ & 181 & \boldsymbol{$12$} & 269 & $19+7\beta$\\
41 & \boldsymbol{$-3$} & 107 & $-3+12\beta$ & \Circle{191} & $11 + 2\beta$ & \Circle{271} & $6-3\beta$\\
43 & $5+3\beta$ & 109 & $-1+6\beta$ & 193 & $18-8\beta$ & 277 & $-8 + 12\beta$\\
\Circle{47} & \boldsymbol{$3$} & 113 & $10-2\beta$ & \Circle{197} & \boldsymbol{$3$} & \Circle{281} & $-2-17\beta$\\
\hline
\end{tabular}
\caption{Prime-index coefficients of 361.2.a.f, with $\beta = \frac{1+ \sqrt{5}}{2}$}
\label{table}
\end{table}

\end{example}

\textbf{Acknowledgements.} We thank Frank Calegari for asking the question that inspired this work and for his encouragement, and we thank Pedro Lemos for helpful conversations related to this project.  The second author acknowledges support from NSF grant DMS-1901867.

\bibliography{eisenstein_congruences}
\bibliographystyle{alpha}
\end{document}